\documentclass[amscd, amssymb, verbatim, 12pt]{amsart}
\usepackage[margin=1.2in]{geometry}
\usepackage{amsmath,setspace,color}
\usepackage{amsfonts,amsthm}
\usepackage{amssymb,enumitem,verbatim}
\usepackage{graphicx}
\usepackage{epstopdf}
\usepackage{ stmaryrd }
\usepackage{dsfont}
\usepackage{indentfirst}
\usepackage{pifont}
\usepackage{tikz}
\usepackage{genyoungtabtikz}
\usepackage{mathrsfs}
\usepackage{fancyhdr}
\usepackage[T1]{fontenc}

\usepackage[foot]{amsaddr}

\newtheorem{Theorem}{Theorem}
\newtheorem{Lemma}[Theorem]{Lemma}

\newtheorem{Problem}{Problem}
\newtheorem{Corollary}[Theorem]{Corollary}

\theoremstyle{definition}
\newtheorem{Definition}[Theorem]{Definition}
\newtheorem*{planA}{Plan A}
\newtheorem*{Example}{Example}
\newtheorem*{Remark}{Remark}
\newtheorem*{Convention}{Convention}

\DeclareMathOperator{\im}{Im}
\let\ker\relax
\DeclareMathOperator{\ker}{Ker}

\DeclareMathOperator{\supp}{supp}
\DeclareMathOperator{\found}{found}

\title[Specht module cohomology and integral designs]{Specht module cohomology and integral designs}
\begin{document}
\author{Ha Thu Nguyen} 
\address{DPMMS, CMS, Wilberforce Road, Cambridge, CB3 0WB, England UK.}
\email{H.T.Nguyen@dpmms.cam.ac.uk}
\thanks{Keywords:  symmetric groups, modular representation theory, integral design theory}
\thanks{Mathematics subject classification(2010): Primary 20C30; Secondary 20E48, 20G40}

\begin{abstract}
We aim to construct an element satisfying Hemmer's combinatorial criterion for $H^1(\mathfrak{S}_n, S^\lambda)$ to be non-vanishing. In the process, we discover an unexpected and surprising link between the combinatorial theory of integral designs and the representation theory of the symmetric groups.\end{abstract}
\maketitle
\section{Introduction}\label{intro}
It is well known that cohomology groups largely control the representation theory. However, these groups are very difficult to compute and very few are known explicitly.  For a Specht module $S^\lambda$ of the symmetric group $\mathfrak{S}_n$, where $\lambda$ is a partition of $n$, the cohomology $H^i(\mathfrak{S}_n,S^\lambda)$ is known only in degree $i = 0$. Further known results about first cohomology groups concern mostly Specht modules corresponding to hook partitions or two-part partitions (Weber~\cite{Weber}). However, the proofs for these involve powerful and complicated algebraic group machinery. Recently, David Hemmer~\cite{Hemmer12} proposed a method that allows (in principle, although it is difficult in practice) to check whether $H^1(\mathfrak{S}_n,S^\lambda)$ is trivial or not, only by means of combinatorics.  While Hemmer's criterion for non-zero $H^1(\mathfrak{S}_n, S^\lambda)$ provides some understanding of these groups, there is currently no effective method for determining when the criterion is fulfilled. Even for the cases where it is known that $H^1(\mathfrak{S}_n, S^\lambda)$ is non-zero, so that the criterion holds, proving directly that the criterion does indeed hold has been difficult. 

In this paper, we aim to construct an element $u$ of the permutation module $M^\lambda$ of the symmetric group $\mathfrak{S}_n$, that satisfies the necessary and sufficient conditions found by Hemmer~\cite{Hemmer12} for the cohomology group $H^1(\mathfrak{S}_n, S^\lambda)$ to be non-zero. In constructing this element for $2$-part partitions, we discover an unexpected and surprising link between the combinatorial theory of integral designs and the representation theory of symmetric groups. A similar, although much more complex, procedure should be sufficient to construct the desired element of Hemmer's criterion for partitions with three or more non-zero parts. There is much hope that this will lead to a classification of partitions labelling the Specht modules $S^\lambda$ such that $H^1(\mathfrak{S}_n,S^\lambda)$ is non-zero, and this is a possible source of future research.

The outline of the paper is as follows. In section $2$, we introduce Hemmer's criterion and collect some relevant known results that will be used throughout. In section $3$, we give a complementary argument to apply the criterion to do some computations and extend Theorems $5.8$ and $5.1$ in Hemmer~\cite{Hemmer12}. This argument highlights the fact that it is difficult to apply Hemmer's method in practice and one needs a novel approach to tackle first degree cohomology groups of Specht modules in a purely combinatorial way. One such approach is to use the theory of integral designs as shown in the second half of section $3$ and section $4$. Finally, we collect in the conclusion some general conjectures posed by Hemmer~\cite{Hemmer12} that could perhaps be attacked by our new approach.
 \section{Hemmer's criterion}
In this section, we describe Hemmer's criterion for non-zero $H^1(\mathfrak{S}_n, S^\lambda)$. Let us first introduce some notations and known results that will be used throughout.

A $\emph{partition}$ of $n$, denoted $\lambda \vdash n$, is a non-increasing string of non-negative integers $\lambda=(\lambda_1,\lambda_2, \dots, \lambda_r)$ summing to $n$. Write $[\lambda]$ for the $\emph{Young diagram}$ for $\lambda$ i.e $[\lambda] = \{ (i,j) \in \mathbb{N}^2\  | j \leq \lambda_i \}$. We adopt the convention of drawing the Young diagram by taking the $i$-axis to run from left to right on the page, and the $j$-axis from top to bottom, and placing a \text{\sffamily \large{x}} at each node. For example,  the partition $(5, 3, 1)$ has the following diagram:
\newcommand{\crossyoung}{\text{\sffamily \large{x}}}
\[
\gyoung(:\crossyoung:\crossyoung:\crossyoung:\crossyoung:\crossyoung,:\crossyoung:\crossyoung:\crossyoung,:\crossyoung)
\]
 A $\lambda$-\emph{tableau} is an assignment of $\{1,2, \dots, n \}$ to the nodes in $[\lambda]$. For example,
 \[
 \gyoung(:1:3:2:5:4,:8:7:9,:6)
 \]
 \noindent
 is a $(5, 3, 1 )$-tableau.

The symmetric group acts naturally on the set of $\lambda$-tableaux. For a tableau $t$ its row stabiliser $R_t$ is the subgroup of $\mathfrak{S}_n$ fixing the rows of $t$ setwise. Say $t$ and $s$ are row equivalent if $t = \pi s$ for some $\pi \in R_s$. An equivalence class is called a $\lambda$-\emph{tabloid}, and the class of $t$ is denoted by $\{ t \}$.  We shall depict $\{t\}$ by drawing lines between rows of t. For instance, 
 \[
 \youngtabloid(12345,678,9)
 \]
 \noindent
 is a $(5, 3, 1)$-tabloid.

  For $R$ a commutative ring, the \emph{permutation module $M^\lambda_R$} is the free $R$-module with basis the set of $\lambda$-tabloids. We drop the subscript $R$ when there is no confusion. If $\lambda=(\lambda_1,\lambda_2, \dots, \lambda_r)$, there is a corresponding \emph{Young subgroup} $\mathfrak{S}_\lambda \cong \mathfrak{S}_{\lambda_1} \times \mathfrak{S}_{\lambda_2} \times \dots \times \mathfrak{S}_{\lambda_r} \leq \mathfrak{S}_n$. The stabiliser of a $\lambda$-tabloid $\{ t \} $ is clearly a conjugate of $\mathfrak{S}_\lambda$ in $\mathfrak{S}_n$, and $\mathfrak{S}_n$ acts transitively on $\lambda$-tabloids, so we have
\[
M^\lambda_R \cong Ind_{\mathfrak{S}_\lambda}^{\mathfrak{S}_n}R.
\]
Since $M^\lambda$ is a transitive permutation module, it has a one-dimensional fixed-point space under the action of $\mathfrak{S}_n$. Let $f_\lambda \in M^\lambda$ denote the sum of all the $\lambda$-tabloids, so $f_\lambda$ spans this fixed subspace.

Note that the definition of $M^\lambda$ as the permutation module of $\mathfrak{S}_n$ on a Young subgroup does not require $\lambda_1 \geq \lambda_2 \geq \dots \geq \lambda_r$ so $M^\lambda$ is also defined for compositions $\lambda$ i.e.\ $\lambda=(\lambda_1, \lambda_2, \dots, \lambda_r)$ such that $\sum_{i=1}^r \lambda_i =n$.
\begin{Example}
\[
n=10, \lambda = (4, 5, 0, 1) \text{ and } M^{(4, 5, 0, 1}) \cong M^{(5, 4, 1)}.
\]
\end{Example}
The Specht module $S^{\lambda}$ is defined explicitly as the submodule of $M^{\lambda}$ spanned by certain linear combination of tabloids, called polytabloids. In characteristic zero the Specht modules $\{S^{\lambda} | \lambda \vdash n\}$ give a complete set of non-isomorphic simple $\mathfrak{S_n}$-modules. James~\cite{James78} gave an important alternative description of $S^\lambda$ inside $M^\lambda$ as the intersection of the kernels of certain homomorphisms from $M^\lambda$ to other permutation modules.

 Let $\lambda=(\lambda_1,\lambda_2, \dots, \lambda_r) \vdash n$, and let $\nu=(\lambda_1,\lambda_2, \dots, \lambda_{i-1}, \lambda_{i} + \lambda_{i+1} -v, v, \lambda_{i+2}, \dots )$.  Define the module homomorphism $\psi_{i,v} : M^\lambda \to M^\nu$ by 
\begin{align*}
\psi_{i,v}(\{t\}) =& \sum \{ \{t_1\} \ |\ \{t_1\} \ \text{agrees with}\  \{t\} \ \text{on all rows except rows}\ i \ \text{and}\ i+1,\\
&\text{and row} \ i+1 \ \text{of}\ \{t_1\} \ \text{is a subset of size}\ v \ \text{in row} \ i+1 \ \text{of} \ \{t\}\}.
\end{align*}

\begin{Theorem}[\textbf{Kernel Intersection Theorem}, James~\cite{James78}] \label{kernel}
Suppose $\lambda \vdash d$ has $r$ non-zero parts. Then
\[
S^\lambda = \bigcap_{i=1}^{r-1} \bigcap_{v=0}^{\lambda_i-1} \ker (\psi_{i,v}) \subseteq M^\lambda.
\]
\end{Theorem}
So given a linear combination of tabloids $u \in M^\lambda$, this gives an explicit test for whether $u\in S^\lambda$. 
\begin{Theorem}[James~\cite{James78}]\label{H0}
Given $t \in \mathbb{Z}$, let $l_p(t)$ be the smallest non-negative integer satisfying $ t < p^{l_p(t)}$. The invariants $Hom_{k\mathfrak{S}_d}(k,S^\lambda)=H^0(\mathfrak{S}_d,S^\lambda)=0$ unless $\lambda_i \equiv -1 \pmod {p^{l_p(\lambda_{i+1})}}$ for all $i$ such that $\lambda_{i+1} \neq 0$, in which case it is one-dimensional.
\end{Theorem}
Generalizing James' work, Hemmer proposed
\begin{Theorem}[Hemmer~\cite{Hemmer12}]\label{ext} Let $p>2$ and $\lambda \vdash d$. Then $Ext^1(k, S^\lambda) \neq 0$ if and only if there exists $u \in M^\lambda$ with the following properties:

(i) For each $\psi_{i,v}: M^\lambda \to M^\nu$ appearing in Theorem~\ref{kernel}, $\psi_{i,v}(u)$ is a multiple of $f_\nu$, at least one of which is a non-zero multiple.  

(ii) There does not exist a scalar $a \neq 0$ such that all the $\psi_{i,v}(af_\lambda-u)$ are zero.

\noindent
If so then the subspace spanned by $S^\lambda$ and $u$ is a submodule that is a non-split extension of $k$ by $S^\lambda$.
\end{Theorem}
\begin{Remark}
We can replace $M^\lambda$ by the Young module $Y^\lambda$ --- the unique indecomposable direct summand of $M^\lambda$ containing the Specht module $S^\lambda$. 
\end{Remark}

In the same paper, Hemmer~\cite{Hemmer12} proposed the following problem:
\begin{Problem}\label{non-vanishing hom}
It is known by an argument of Andersen (Proposition $5.2.4$~\cite{Hem09}) that for $\lambda \neq (d)$, if $H^0(\mathfrak{S}_d,S^\lambda) \neq 0$ then $H^1(\mathfrak{S}_d, S^\lambda) \neq 0$. For each such $\lambda$ (given by Theorem~\ref{H0}) construct an element $u \in M^\lambda$ as in Theorem~\ref{ext}.
\end{Problem} 
\begin{Remark}
We note that for $\lambda$ such that $H^0(\mathfrak{S}_d,S^\lambda) \neq 0$, condition (ii) in \textbf{Theorem~\ref{ext}} is implied by condition (i). Hence, one natural approach to \textbf{Problem~\ref{non-vanishing hom}} is
\end{Remark}
\begin{planA}
Find a $u \in M^\lambda$ such that all the $\psi_{i,v}(u)$ vanish except for precisely one value of $(i, v)$. 
\end{planA}

Finally, many of the problems involving Specht modules that arise depend upon whether or not the prime characteristic $p$ divides certain binomial coefficients. We collect several relevant lemmas below.

\begin{Lemma}[James~\cite{James78}, Lemma 22.4]\label{lem224}
Assume that 

$a = a_0 + a_1p + \dots + a_rp^r (0 \leq a_i < p),$

$b = b_0 + b_1p + \dots + b_rp^r  (0 \leq b_i <p).$

\noindent
Then $\binom{a}{b} \equiv \binom{a_0}{b_0}\binom{a_1}{b_1}\dots \binom{a_r}{b_r}$ $\pmod{p}$. In particular, $p$ divides $\binom{a}{b}$ if and only if $a_i < b_i$ for some $i$.
\end{Lemma}

\begin{Corollary}[James~\cite{James78}, Corollary 22.5]\label{cor225}
Assume $a \geq b \geq 1$. Then all the binomial coefficients $\binom{a}{b}$, $\binom{a-1}{b-1}, \dots, \binom{a-b+1}{1}$ are divisible by $p$ if and only if
\[
a-b \equiv  (-1)\pmod{p^{l_p(b)}}.
\]
\end{Corollary}

\begin{Lemma}[Kummer~\cite{Kummer}, p.116]\label{lemkum}
The highest power of a prime $p$ that divides the binomial coefficient $\binom{x+y}{x}$ is equal to the number of carries that occur when the integers $x$ and $y$ are added in $p$-ary notation.
\end{Lemma}

\section{New results and introduction to design theory}
\subsection{Two-part partitions and integral designs.}
In this section, we specialize to $2$-part partitions. In this case, $\lambda$-tabloids are determined by the second row. 

\begin{Example} We can represent 
 \[
 \youngtabloid(12345,678)
 \]
simply by $\overline{6\ 7 \ 8}$.
\end{Example}

In~\cite{Hemmer12}, Hemmer followed \textbf{Plan A} to find explicit u's for $\lambda = (p^a, p^a)$ and $\lambda = (p^b-1,p^a)$ for $b >a$.  We will give a complementary argument to shed light on why this natural approach is particularly suitable for the values of $\lambda$ chosen by Hemmer: they belong to a subclass of the class of $2$-part partitions where the second part is a $p$-power. Our argument also extends Theorems $5.8$ and $5.11$ in Hemmer~\cite{Hemmer12}.

\begin{Theorem}\label{2-p-power}
Let $k$ have characteristic $p \geq 3$. Then 

(1) $H^1(\mathfrak{S}_d, S^{(rp^n,\  p^n)}) \neq 0$ for any $n\geq 1$, $r \geq1$, $p \nmid (r+1)$ and $p \nmid r$.

(2) $H^1(\mathfrak{S}_d, S^{(a,\  p^n)}) \neq 0$ for any $a \geq p^n$ and $a\equiv (-1) \pmod{p^{n+1}}$.
\end{Theorem}

\begin{proof}

Suppose for a partition $(a,b)$ in the class of partitions considered in $(1)$ and $(2)$ in the theorem, there is an explicit $u \in M^{(a, b)}$, 
 such that for each 

\begin{align*}
\psi_{1,v}: &M^{(a,\ b)} \to M^{(a+b-v,\ v)}\\
                & \overline{i_1 \dots i_b} \mapsto \sum_{\{j_1,\dots,j_v\} \in [b]^{(v)}}\overline{i_{j_1} \dots {i_{j_v}}}
\end{align*}
where $0 \leq v \leq b-1, \psi_{1,v}(u)=\lambda_vf_v$, where $\lambda_v \neq 0$ for some $v$, $f_v=f_{(a+b-v,\ v)}$, $[b] = \{1, 2, \dots, b \}$ and for a set $X, X^{(v)}$ denotes the set of all subsets of size $v$ of X.

Let $u = \sum{\lambda_{i_1...i_b}}\overline{i_1\dots i_b}$ where the sum is over all the subsets of size $b$ of $\{1,2,\dots$ $a+b\}$. Then for a fixed $v$, $1\leq v \leq b-1$,
\[
\psi_{1,v}\left(\sum{\lambda_{i_1...i_b}}\overline{i_1\dots i_b}\right) = \sum{\lambda_{i_1...i_b}(\overline{i_1 \dots i_v} + \dots) = \lambda_v(\overline{1 \dots v} + \dots)} = \lambda_v f_v,
\]
where both sums are over all the $b$-sets of $\{1, 2, \dots, a+b\}$. Counting the total number of tabloids involved on both sides, i.e. equating the sum of coefficients on both sides,  we have
$\binom{a+b}{v}\lambda_v = \binom{b}{v} \sum{\lambda_{i_1...i_b}}$. Replace $v$ by $b-v$ to get

\[
\binom{a+b}{b-v}\lambda_{b-v} = \binom{b}{b-v} \sum{\lambda_{i_1...i_b}}
\]
\begin{align*}
&\Rightarrow \frac{(a+b)(a+b-1)\dots (a+v+1)}{(b-v)!} \lambda_{b-v}=\frac{b(b-1) \dots (b-v+1)}{v!}\sum{\lambda_{i_1...i_b}}\\
&\Rightarrow (a+b)(a+b-1)\dots (a+v+1) v! \lambda_{b-v} = b(b-1) \dots (b-v+1) (b-v)!  \sum{\lambda_{i_1...i_b}}\\
&\Rightarrow (a+b)(a+b-1)\dots (a+v+1) \lambda_{b-v} = b(b-1) \dots (v+1) \sum{\lambda_{i_1...i_b}} \label{star} \tag{$\star$}
\end{align*}

\underline{\emph{Case $1: a = r p^n, \ b = p^n,\text{ and p }\nmid r.$}}

\noindent
Recall that for an integer $m$, $v_p(m)$ is the greatest integer $t$ such that $p^t | m$.

\noindent
\underline{Claim:} For p an odd prime,
\[
v_p((r+1)p^n-i) = v_p(p^n-i) \quad \forall 1\leq i \leq p^n-1.
\]

\noindent
\underline{Proof of claim:}

This follows from one of the well-known properties of non-archimedean valuation:
\[
v_p(m) \neq v_p(n) \Rightarrow v_p(m+n) = \text{ min} \{ v_p(m),v_p(n)\}.
\]

Next, we work with the case $\lambda \in \mathbb{Z}$ only. From the above claim and equation~\eqref{star}, $v_p(\lambda_{b-v}) = v_p(\sum{\lambda_{i_1...i_b}})$ for all $1\leq v \leq b-1$, i.e.\ $v_p(\sum{\lambda_{i_1...i_b}}) = v_p(\lambda_1) = v_p(\lambda_2) = \dots = v_p(\lambda_{b-1}) $.
It is also easy to see from~\eqref{star} that $\lambda_{b-v} \equiv \sum{\lambda_{i_1...i_b}}\pmod {p}$ for all $1\leq v \leq b-1$, i.e.\ $\sum{\lambda_{i_1...i_b}} \equiv \lambda_1 \equiv \lambda_2 \equiv \dots \equiv \lambda_{b-1} \pmod{p}$ i.e.\ $\sum{\lambda_{i_1...i_b}} = \lambda_1 = \lambda_2 = \dots = \lambda_{b-1}$ in $k$.  We believe these necessary conditions are also sufficient: it is likely that we could find a $u$ for each of the values of $\sum{\lambda_{i_1...i_b}}$ mod $p$. Here we give $u$ in the case $\sum{\lambda_{i_1...i_b}} \equiv 0 \pmod p$: let $u = \sum_{i=0}^{p^n-1}(i+1)v_i$, where 
\[
v_i = \sum\{\{t\} \in M^{(rp^n,p^n)} \mid \text{exactly } i \text{ of } \{1, 2, \dots, p^n-1\} \text{ lie in row two of } \{t\}\},
\]
\noindent
for $0 \leq i \leq p^n-1$.
In addition to a similar argument to Hemmer~\cite{Hemmer12}, Lemma~\ref{lem224} and Corollary~\ref{cor225} show that $\psi_{1,0}(u)=c \cdot \emptyset \text{  where }p \nmid c$, and $\psi_{1,i}(u) = 0$ for all $i \geq 1$:

$\bullet$ Counting the number of tabloids that appear in the sum defining $v_i$, there are $\binom{p^n - 1}{i}$ choices for the row two entries from $\{1, 2, \dots, p^n -1\}$ and $\binom{rp^n +1}{p^n - i}$ for the remaining entries from $\{p^n, p^n +1 ,\dots, (r+1)p^n \}$. Hence, $\psi_{1, 0}(v_i) = \binom{p^n - 1}{i} \binom{rp^n+1}{p^n-i} \emptyset$. By Lemma ~\ref{lem224}, $\psi_{1, 0}(v_i) = 0$ for all $i \notin \{0, p^n -1\}$. Thus, $\psi_{1, 0}(u) = \psi_{1, 0}(v_0) + p^n\psi_{1,0}(v_{p^n - 1}) = \binom{p^n - 1}{0}\binom{rp^n+1}{p^n} \emptyset = c \cdot \emptyset$ for some $c$ with $p \nmid c$. 

$\bullet$
 For $1 \leq t \leq s < p^n$, the coefficient of 
 \[
 \overline{1, 2, 3, \dots , t, p^n, p^n +1, \dots p^n +s -t -1} \in M^{((r+1)p^n - s, s)}
 \]
 \noindent
 in $\psi_{1, s}(v_i)$ is $\binom{p^n - 1 -t}{i - t}\binom{rp^n -s +t +1}{p^n -s +t -i}$. This follows from counting the number of tabloids in the sum defining $v_i$ that contribute to the coefficient when we evaluate $\psi_{1, s}(v_i)$: such a tabloid must have $\{1, 2, \dots, t\}$ in the second row, so there are $\binom{p^n - 1- t}{i -t}$ choices for the remaining entries from $\{1, 2, \dots, p^n -1 \}$ and $\binom{rp^n -s + t + 1}{p^n - s + t- i}$ choices for the remaining entries from $\{p^n, p^n+1, \dots, (r+1)p^n\}$.

\indent 
Now let $A_{s, t}$ be the coefficient of 
  \[
 \overline{1, 2, 3, \dots , t, p^n, p^n +1, \dots p^n +s -t -1} \in M^{((r+1)p^n - s, s)}
 \]
\noindent
in $\psi_{1, s}(u)$. Then
\[
A_{s, t} = \sum_{m =t}^{p^n -1}(m+1)\binom{p^n -1 -t}{m-t}\binom{rp^n -s +t +1}{(r-1)p^n + m+1}.
\]

 \noindent
 \underline{Claim}

$\bullet$ For $1 \leq s < p^n$, $A_{s, s-1} \equiv 0 \pmod p$.

$\bullet$ For $1 \leq t \leq s < p^n$, we have
\[
A_{s, t} - A_{s, t-1} = \binom{(r+1)p^n - s- 1}{rp^n -1} \equiv 0 \pmod p.
\]
 \underline{Proof of claim:}
 
 When $t = s-1$, the second binomial coefficient in each term in the sum defining $A_{s, s-1}$ is $\binom{rp^n}{(r-1)p^n + m + 1}$, which is congruent to zero except for the last term $m = p^n -1$ by Lemma~\ref{lem224}, in which case the factor $m+1$ in the term is zero.
 
 To prove the second bullet point, apply the identity $\binom{rp^n -s +t +1}{(r-1)p^n + m + 1} = \binom{rp^n -s +t}{(r-1)p^n + m}+\binom{rp^n -s +t}{(r-1)p^n + m + 1}$ to the second coefficient in the defining sum. Expand out and collect term to obtain:
 \begin{align*}
 A_{s, t} &= (t+1)\binom{p^n - t -1}{0} \binom{rp^n -s +t}{(r-1)p^n +t}\\
 &+ \sum_{w =t}^{p^n -2}\left[ (w+1) \binom{p^n - t -1}{w - t} + (w + 2)\binom{p^n -t -1}{w - t +1}\right]\binom{rp^n - s + t}{(r-1)p^n + w + 1}\\
 &+ p^n\binom{p^n - t- 1}{p^n -t -1}\binom{rp^n -s + t}{rp^n}.
 \end{align*}
\noindent
Finally, replace each $(w + 1)\binom{p^n -t -1}{w - t} + (w + 2)\binom{p^n -t -1}{w - t +1}$ in the above equation by
\[
(w + 1)\binom{p^n -t}{w -t +1} + \binom{p^n - t - 1}{w -t +1}
\]
\noindent
and subtract off
\[
A_{s, t-1} = \sum_{w = t-1}^{p^n -1}(w + 1)\binom{p^n -t}{w -t +1}\binom{rp^n -s + t}{(r-1)p^n + w +1}
\]
\noindent
to obtain
\[
A_{s, t} - A_{s, t-1} = \sum_{w = t}^{p^n - 1}\binom{p^n - t -1}{w - t}\binom{rp^n -s +t}{(r-1)p^n + w} = \binom{(r+1)p^n - s- 1}{rp^n -1}.
\]

The last equality follows from the identity
\[
\sum_{k}\binom{l}{m + k}\binom{s}{n +k} = \binom{l+s}{l -m +n}, 
\]
\noindent
for $ l \geq 0$ and integers $m, n$ in \cite{concrete}, (2.53). Here we take $l = p^n -t -1, s= rp^n -s + t, k =w, m =-t$, and $n = (r-1)p^n$.

By Lemma~\ref{lemkum}, $A_{s, t} - A_{s, t-1} \equiv 0 \pmod p$: expanding $rp^{n} -1$ in $p$-ary notation, all the digits corresponding to $p^0, p, \dots, p^{n-1}$ are $p -1$ and some digit of $p^n- s $ corresponding to $p^0, p, \dots, p^{n-1}$ is non-zero as $1 \leq s <p^n$. Thus adding them together in $p$-ary notation will always result in at least one carry. Thus, the claim is true.

   Finally, $\psi_{1,p^n-1}(f_{(rp^n,\ p^n)}) = \binom{rp^n+1}{1}f_{(rp^n+1,\ p^n-1)}$ so condition $(2)$ of Theorem~\ref{ext} is satisfied, too, and we are done by symmetry (c.f. Remark $5.3$ in~\cite{Hemmer12}).
   

    Later we will see that the necessary conditions for $u$ for  a general $2$-part partition turn out to be sufficient over $\mathbb{Z}$.

\underline{\emph{Case $2: b = p^n, a\equiv (-1) \pmod{p^{n+1}}$}}.

\noindent
Since $a\equiv (-1) (p^{l_p(b)})$, and $b = p^n $, we have

$a = (p-1) + (p-1)p + \dots + (p-1) p^n + a_{n+1}p^{n+1} + \dots$, where $a_m \geq1$ for some $m \geq (n+1)$.

\noindent
Thus, 
\begin{align*}
v_p(a+2) &= v_p(1)\\
v_p(a+3) &= v_p(2) \\
&\dots \\
v_p(a+b) &= v_p(b-1) 
\end{align*}

\noindent
Therefore, \eqref{star} implies 
\[
v_p(a+v+1) + v_p(\lambda_{b-v}) = v_p(b) + v_p(\sum{\lambda_{i_1...i_b}}).
\]

\noindent
Since $1 \leq v \leq b-1$ and $a+1 \equiv 0 \pmod{p^{n+1}}$, $v_p(a + 1 + v) = v_p(v)$ so we have
\[
v_p(v) + v_p(\lambda_{b-v}) = v_p(b) + v_p(\sum{\lambda_{i_1...i_b}}). \label{2star} \tag{$\star$ $\star$}
\]

\noindent
Now as $ b = p^n$, $v_p(v) < v_p(b)$ for all $1 \leq v \leq b-1$. Thus, for~\eqref{2star} to hold, $v_p(\lambda_{b-v}) > 0$ i.e. $\lambda_{b-v} = 0$ in $k$ for all $v$. This means that there is \emph{no} non-trivial fixed point in $\im\psi_{1,b-v}$ for $1 \leq v \leq b-1.$ Hence, if we can choose a $u$ such that $\psi_{1,v} (u) = \lambda_v f_v$ for $0 \leq v \leq b-1$, $\psi_{1,v}(u)$ must vanish for all $1\leq v\leq b-1$. 

Note that this is not the case if $b$ is not a $p$-power: if $b = p^r \tilde{b}$ where $\tilde{b} \geq 2$ and $ p \nmid \tilde{b}$, then a similar argument to the above argument shows that there is no non-trivial fixed point in any of the images of 
\begin{align*}
&\psi_{1,b-1}, \psi_{1,b-2}, \dots, \psi_{1, b-p^r+1},\\
&\psi_{1,b-p^r-1}, \psi_{1, b-p^r-2}, \dots, \psi_{1, b-2p^r+1},\\
&\dots\\
&\psi_{1,b-(\tilde{b}-1)p^r -1},\psi_{1,b-(\tilde{b}-1)p^r -2}, \dots, \psi_{1,b-\tilde{b}p^r -1}.
\end{align*}

\noindent
However, $v_p(p^r) = v_p(2p^r) = \dots = v_p((\tilde{b}-1)p^r) = v_p(\sum{\lambda_{i_1...i_b}})$, so the images of $\psi_{1,p^r}, \psi_{1,2p^r},$ $\dots, \psi_{1,(\tilde{b}-1)p^r}$ may contain a non-trivial fixed point. 

\noindent

This is the reason why \textbf{Plan A} is particularly practicable if $b$ is a $p$-power.  Our candidate $u$ when $b = p^n$ is
\[
u = \sum \{\{t \} \in M^{(a,\ p^n)} | 1, 2, \dots, p^n \ \text{appear in the first row of} \{t\}\}.
\]

\noindent
Then $\psi_{1,0}(u) = \binom{a}{p^n} \overline{\emptyset}$ and $\psi_{1,v} (u)= \binom{a-v}{p^n-v} \cdot m$, where $m \in M^{(a,p^n)}, 1\leq v \leq p^n-1$. Since $a \equiv (-1) \pmod{p^{n+1}}$, $p \nmid \binom{a}{p^n}$ and $p | \binom{a-v}{p^n-v}$ for $1\leq v \leq p^n-1$ by Lemma ~\ref{lem224} and Corollary~\ref{cor225}. We note that  $H^0(\mathfrak{S}_{a+p^n},S^{(a,\ p^n)}) \neq 0$, by Theorem~\ref{H0}, so $u$ satisfies the condition of \textbf{Theorem~\ref{ext}} and we are done.
\end{proof}

It can be seen that \textbf{Plan A} does not seem to work for general $2$-part partitions. To progress further, we need to have a different approach. If we work over $\mathbb{Z}$ instead of $k$, the necessary conditions for $\lambda_v \neq 0$ turn out to be sufficient conditions as well. In fact, they arise from fundamental and  until recently  poorly understood combinatorial structures called \textbf{\emph{integral designs}} (or $t$-designs as in Dembowski~\cite{Dembowski}), and the $u$ we want over $\mathbb{Z}$ turns out to be an $(a+b,b,\lambda_0,\dots, \lambda_{b-1})$-design.
\begin{Definition}
Given integers $t, v, l, \lambda_0, \dots, \lambda_t$, where $v\geq 1$ and $0 \leq t, l \leq v$, let $V= \{1, 2, \dots, v\}$, $X = \{ x | x\subseteq V\}$, and $V_l = \{x | x\subseteq V, |x| = l \} = [v]^{(l)}$. The elements of $X$ are called \emph{blocks} and those in $V_l$ are called \emph{$l$-blocks} or \emph{blocks of size $l$}. An \emph{integral $(v, l, \lambda_0, \dots, \lambda_t)$-design} associates integral multiplicities $c(x)$ to $l$-blocks $x$ and zero to all other blocks such that
\[
\hat{c}(y) := \sum_{x\supseteq y} c(x)=\lambda_s, \quad \text{if}\quad |y| = s \leq t. 
\]
If all the parameters $\lambda_i$'s are zero, it is called a \emph{null design}.
\end{Definition}
\begin{Theorem}[Graver-Jurkat~\cite{GJ73}, Wilson~\cite{W73}]\label{design}
There exists an integral $(v, l, \lambda_0, \lambda_1,$ $\dots, \lambda_t)$-design if and only if $\lambda_{s+1} = \frac{l-s}{v-s}\lambda_s$ for $0\leq s<t$. 
\end{Theorem}
Graver-Jurkat developed a method for constructing a design with prescribed parameters $\lambda_0, \dots, \lambda_t$ satisfying the conditions of Theorem~\ref{design}. We will outline their construction below.

Fixing $t, v$, and $l$, the set of all integral $(v, l, \lambda_0, \dots, \lambda_t)$-designs form a module $C_{tl}(v)$ over $\mathbb{Z}$ and the null designs form a submodule $N_{tl}(v)$. Let $G=(G_{xy})_{x, y \in X}$ be the inclusion matrix on X i.e. 
\[
G_{xy} = 
\begin{cases}
1, &\text{if $x \subseteq y;$}\\
0, &\text{otherwise.}
\end{cases}
\]

\noindent
The transform $\hat{c} = Gc$ is defined by
\[
\hat{c}(u) = (Gc)(u) = \sum_{x \supseteq u}c(x).
\]

By partitioning $X$ into blocks of subsets of size $0 \leq i \leq v$, the vectors $c$ and $d = \hat{c}$ split into blocks $c_l$ (the restriction of $c$ to $V_l$) and $d_t$ (the restriction of $d$ to $V_t$). Furthermore, the matrix $G$ or $G(v)$ splits into blocks $G_{tl}$ or $G_{tl}(v)$ such that
\[
d_t = \sum_{l=0}^{v}G_{tl}c_l, \quad \text{for} \  0\leq t \leq v.
\]

\begin{Lemma}[Graver-Jurkat~\cite{GJ73}]\label{matrixId}
For $s \leq h \leq l$,
\[
\binom{l-s}{h-s}G_{sl} = G_{sh}G_{hl}.
\]
\end{Lemma}
\begin{Theorem}[Graver-Jurkat~\cite{GJ73}]\label{fundamental}
Let $0 \leq t \leq l \leq v$. Suppose $c \in \mathbb{Z}_{+}^{X}$ has constant block size $l$, i.e.\ $c(x) = 0 \ \forall x$ with $|x| \neq l$, 
and $(\hat{c})_{t} = \lambda_t e_t$, where $e_t$ is the vector which has components one for each $t$ block. Then $c$ is an integral $(v, l, \lambda_0, \dots, \lambda_t)$-design. Furthermore,
\[
\lambda_s = \frac{\binom{v-s}{t-s}}{\binom{l-s}{t-s}}\lambda_t, \quad \text{for} \ s = 0,1, \dots, t.
\] 
\end{Theorem}
\begin{Theorem}
\[
rank \  G_{tl}(v) = 
\begin{cases}
&\binom{v}{t}, \quad \text{when} \ t\leq l \leq v - t,\\
& \binom{v}{l}, \quad \text{when} \ v - t\leq l \leq v, \ t \leq l.
\end{cases}
\]
\end{Theorem}

\begin{Corollary}\label{dim}
 dim$N_{tl}(v) = \binom{v}{l} - \binom{v}{t}$ and dim$C_{tl}(v) = dim N_{tl}(v) + 1$ for $t \leq l \leq v-t$.
\end{Corollary}

\begin{Definition}(Graver-Jurkat~\cite{GJ73})\label{support}
The \emph{support} of an element $c \in \mathbb{Z}^{X}$ is the collection of elements of $X$ which have non-zero multiplicities:
\[
\supp(c) = \{x | c(x) \neq 0 \}.
\]
\noindent
The \emph{foundation} of an element $c \in \mathbb{Z}^{X}$ is the union of the blocks in its support:
\[
\found(c) = \bigcup_{x: c(x) \neq 0} x.
\]
\end{Definition}

\noindent
\begin{Lemma}[Graver-Jurkat~\cite{GJ73}]\label{foundationSize}
If $c \in C_{tl}(v) \setminus N_{tl}(v)$, then $\found(c) = V$. On the other hand, if $c \in N_{tl}(v)$, then $|\found(c)| \geq t+l+1$.
\end{Lemma}

\begin{Definition}\label{convolution}(Graver-Jurkat~\cite{GJ73})
The \emph{convolution} $\star$ on $\mathbb{Z}+^X$ is defined as follows
\[
(c \star d)(z) = \sum_{x+y = z} c(x)d(y),
\]
\noindent
where $x, y, z \in X, c, d \in \mathbb{Z}_+^X$ and $+$ is the Boolean sum or symmetric difference.
\end{Definition}
\begin{Convention}
$N_{-1l}(v) = \{c | c \in \mathbb{Z}^X \  \text{and} \ c(x) = 0 \ \text{whenever} \ |x| \neq l \}$.
\end{Convention}

\begin{Definition}
For $x \in X$,  let $\delta_{x} \in \mathbb{Z}_{+}^{X}$ be the indicator function of $x$, i.e.
\[
\delta_x (y) =
\begin{cases}
&1, \quad \text{if}\ y = x,\\
&0, \quad \text{otherwise.}
\end{cases}
\]
\noindent
If $x \subseteq V$, define the \emph{extension of $c$ by $x$} by $c \star \delta_x$.
\end{Definition}
\begin{Definition}
If $q$ and $r$ are distinct points in $V$, let $d_{qr} = \delta_{\{q\}} - \delta_{\{r\}}$. Define the \emph{suspension of $c$ by $q$ and $r$} by $c \star d_{qr}$.
\end{Definition}
\begin{Theorem}[Graver-Jurkat~\cite{GJ72}]
If $c \in C_{t,l_1}(v)$, $d \in C_{t, l_2}(v)$ and $|x \cap y|$ is fixed whenever $c(x)d(y) \neq 0$, then $c \star d \in C_{t, l_3}(v)$
\end{Theorem}
\begin{Theorem}[Graver-Jurkat~\cite{GJ73}]\label{null-convolution}
Let $-1 \leq t \leq l$, $-1 \leq s \leq h$, $0 \leq h, l \leq v$, $u = \found(c)$ and $w = \found(d)$. If $c \in N_{tl}(v)$, $d \in N_{sh}(v)$ and $u \cap w = \emptyset $, then $c \star d \in N_{t+s+1, k+h}(v).$
\end{Theorem}
\begin{Corollary}[Graver-Jurkat~\cite{GJ73}]\label{suspension}
If $c \in N_{tl}(v)$, $|x| = h$ and $\found(c) \cap x =\emptyset$, then $c \star \delta_x \in N_{t,l+1}(v)$. Also, if $\found(c) \cap \{q, r \} = \emptyset$, then $c \star d_{qr} \in N_{t+1,l+1}(v)$.
\end{Corollary}
\begin{Definition}
Let $W = \{1, 2, \dots, , v-1 \}$ and $Y = \{x \mid x \subseteq W \}$ where $v >1$. Define $\phi : \mathbb{Z}^X \to \mathbb{Z}^Y$ by $(\phi c)(x) = c(x + v)$, where $x \in Y$ and  $v$ denotes $\{v\}$.

\end{Definition}
\begin{Lemma}[Graver-Jurkat~\cite{GJ73}]\label{map}
$\phi$ commutes with $G$:
\[
G(v-1)\phi = \phi G(v),
\]
and
\begin{align*}
\phi(C_{tl}(v)) &\subseteq C_{t-1,l-1}(v),\\
\phi |_{N_{tl}(v)}(N_{tl}(v)) &= N_{t-1,l-1}(v).
\end{align*}
\end{Lemma}
\begin{Definition}[Graver-Jurkat~\cite{GJ73}]\label{pod}
Assume that $t < l < v-t$. Let $u$ be any $(l+t+1)$-subset of $V$. Label $2t+2$ of the points in $u$ as $p_0, p_1, \dots, p_t, q_0, q_1, \dots, q_t$ and let $x$ be the subset of the $l - t - 1$ remaining points in $u$. $d$ is a \emph{$t,l$-pod}  if
\[
d = d_{p_0q_0} \star \dots \star d_{p_t q_t} \star \delta_x.
\]
\end{Definition}
\begin{Theorem}[Graver-Jurkat~\cite{GJ73}]\label{nullBasis}
For $0 \leq t, l \leq v$ and $v \geq 1$, $N_{tl}(v)$ has a module basis consisting of designs with foundation size $l + t + 1$; in fact it has a module basis consisting of t,l-pods.
\end{Theorem}
\begin{Theorem}[Graver-Jurkat~\cite{GJ73}]\label{surject}
Let $0 \leq t < l < v-t$. Then $G_{t+1,l}(N_{t,l}) = N_{t,t+1}$.
\end{Theorem}
\begin{Theorem}[Graver-Jurkat~\cite{GJ73}, Wilson~\cite{W73}]\label{mainDesign}
Let $t, v, l, \lambda_0, \dots \lambda_t$ be integers where $v \geq 1$ and $0\leq t, l \leq v$. There exists an integral $(v, l, \lambda_0,\dots, \lambda_t)$-design if and only if
\[
\lambda_{s+1} = \frac{l-s}{v-s}\lambda_s,\quad \text{for} \ 0\leq s <t.
\]
\end{Theorem}
\begin{proof}
The necessary conditions for $t \leq l$ follow from Theorem~\ref{fundamental}. If $t >l$, $\lambda_s = 0$ by definition for $l < s \leq t$.

We will prove the sufficient conditions by induction on $v$. If $t = 0$, the set of conditions is vacuous and $\lambda_0\delta_x$ is a $(v, l, \lambda_0)$-design for $|x| = l$. Assume that the conditions are sufficient for some $t \geq 0$, and that $\lambda_0, \dots, \lambda_{t+1}$ satisfy the conditions. Then there exists an integral $(v, l, \lambda_0, \dots, \lambda_t)$-design $c'$. We would like to alter $c'$ so as to make it an integral $(v, l, \lambda_0, \dots, \lambda_{t+1})$-design.

If $l < t$ or $l > v-t$, then $C_{tl}(v)$ is a one-dimensional space spanned by $e_l$. In this case, $c' = \alpha e_l$, and hence is a $(v, l, \lambda_0, \dots, \lambda_t, \lambda_{t+1}')$-design. We need only check that $\lambda_{t+1} = \lambda_{t+1}'$, which is a straightforward computation.

Now we turn to the case $t < l < v-t$. We have
\[
G_{tl}c_l' = \lambda_t e_t. \label{one} \tag{1}
\]
By Lemma~\ref{matrixId},
\[
G_{tl} = \frac{1}{l-t}G_{t,t+1}G_{t+1,l},
\]
\noindent
and one may easily compute that 
\[
e_t = \frac{1}{v-t}G_{t,t+1}e_{t+1}.
\]
\noindent
Substituting these into \eqref{one} yields
\[
G_{t,t+1}G_{t+1,l}c_l' = G_{t,t+1}\frac{l-t}{v-t}\lambda_t e_{t+1} = G_{t,t+1}\lambda_{t+1}e_{t+1}.
\]
\noindent
Let $d'$ be the extension by zero of $(G_{t+1,l}c_l' - \lambda_{t+1}e_{t+1})$, it is clear that $d' \in N_{t,t+1}$. By Theorem~\ref{surject}, there exist $d \in N_{tl}$ such that $G_{t+1,l}d_l = d_{t+1}'$. Finally, if $c = c' -d$, then $c$ has constant block size $l$ and
\[
G_{t+1,l}c_l = G_{t+1,l}c_l' - d_{t+1}' = \lambda_{t+1}e_{t+1}.
\]
\noindent
It follows from Theorem~\ref{fundamental} that $c$ is a $(v, l, \lambda_0, \dots, \lambda_t, \lambda_{t+1})$-design.
\end{proof}
Turning our attention to the theory of Specht module cohomology, we can now state and prove our main theorem for $2$-part partitions:
\begin{Theorem}\label{main}
Over an algebraically closed field $k$ of odd characteristic, we can construct an $u \in M^{(a, b)}$, such that for each 
\begin{align*}
\psi_{i,v}: &M^{(a, b)} \to M^{(a+b-v, v)}\\
                & \overline{i_1 \dots i_b} \mapsto \sum_{\{j_1,\dots,j_v\} \in [b]^{(v)}}\overline{i_{j_1} \dots {i_{j_v}}}
\end{align*}
where $0 \leq v \leq b-1, \psi_{1,v}(u)=\lambda_vf_v$ and $\lambda_v \neq 0$ for some $v$ 
, $f_v=f_{(a+b-v, v)}$, $[b]=\{1, 2, \dots, b\}$ and for a set $X, X^{(v)}$ denotes the set of all subsets of size $v$ of X. 

Therefore, \textbf{Problem \ref{non-vanishing hom}} can be solved in the case of $2$-part partitions.
\end{Theorem}
\begin{proof}
Note that such a $u$ is equivalent to an integral $(a+b, b, \lambda_0, \dots, \lambda_{b-1})$-design with the parameters $\lambda_i's$ satisfying the sufficient conditions in Theorem~\ref{mainDesign} \emph{and} $\lambda_i \not\equiv 0 \pmod{p}$, for some $i$, since we are working over $k$. How do we find such $\lambda_i's$? We need $\lambda_{s+1} = \frac{b-s}{a+b-s}\lambda_s \ \forall 0\leq s < b-1$, i.e.
\begin{align*}
\lambda_1 &= \frac{b}{a+b}\lambda_0,\\
\lambda_2 &= \frac{b-1}{a+b-1}\frac{b}{a+b}\lambda_0,\\
&\dots
\end{align*}
\begin{align*}
\lambda_{s+1} &=\frac{b-s}{a+b-s} \dots \frac{b}{a+b}\lambda_0 \\
&=\frac{b!}{(b-s-1)!}\frac{(a+b-s-1)!}{(a+b)!}\lambda_0\\
&=\frac{\binom{a+b-s-1}{a}}{\binom{a+b}{a}}\lambda_0,\\
&\dots
\end{align*}
\noindent
Let $d = \text{min}\{ v_p(\binom{a+b-s-1}{a})\}$, where the minimum is taken over the set $\{0, 1, \dots, b-1\}$, and put $\lambda_s = p^{-d}\binom{a+b-s-1}{a}$.
%
 Note that our choice of $\lambda_s$ ensures that all the $\lambda_s's$ are integers. Construct an integral $(a+b, b, \lambda_0,\dots, \lambda_{b-1})$-design $c$ as in Theorem~\ref{mainDesign} and let 
\[
u = \sum_{y \in M^{(a,b)}}c(y)y.
\]
\noindent
Then $\psi_{1,v}(u) = \lambda_v f_v$ by Theorem~\ref{mainDesign} and $\lambda_0 \not\equiv 0 \pmod{p}$ or $\lambda_{s_0} \not\equiv 0 \pmod{p}$ by our choice of $\lambda_0$. 
\end{proof}
\section{Three-part partitions and beyond.}
The next natural step would be to generalize Graver-Jurkat's method to solve our problem for arbitrary partitions. For $3$-part partitions $\lambda= (a,b,c)$, $\lambda$-tabloids are determined by the second and the third rows. 

\noindent
\emph{Example.} Let $\overline{s}$ denote a set of size $s$. Then an element of $M^{(a, b, c)}$ can be represented as $\frac{\ \overline{b}\ }{\ \overline{c}\ }$.

We have
\begin{align*}
\psi_{1,v}: M^{(a, b, c)} &\mapsto M^{(a+b-v, v, c)}\\
\frac{\ \overline{b}\ }{\ \overline{c}\ } &\mapsto  \sum_{\overline{v} \subseteq \overline{b}}\frac{\ \overline{v}\ }{\ \overline{c}\ }
\end{align*}

\noindent
for $0 \leq v \leq b-1$, and
\begin{align*}
\psi_{2,w}: M^{(a,b,c)} &\mapsto M^{(a,b+c-w,w)}\\
\frac{\ \overline{b}\ }{\ \overline{c}\ } &\mapsto \sum_{\overline{w} \subseteq \overline{c}} \frac{\ \overline{b} \cup (\overline{c}\smallsetminus \overline{w})\ }{\ \overline{w}\ }
\end{align*}

\noindent
for $0 \leq w \leq c-1$.

We want to find a $u = \sum_{\overline{b}, \overline{c}}\alpha(\overline{b},\overline{c})\frac{\ \overline{b}\ }{\ \overline{c}\ }$ such that $\psi_{1,v}(u) = \lambda_{v}^1 f_{a+b-v, v,c}$ and $\psi_{2,w}(u) = \lambda_{v}^2 f_{a,b+c-w,w}$, where $\alpha(\overline{b}, \overline{c}) \in k$, $\lambda_v^1$ and $\lambda_{w}^2$ not all zero for $0 \leq v \leq b-1$ and $0 \leq w \leq c-1$.

Note that if we fix the third row of the tabloids involved in $u$, then we have again an integral $(a+b, b, \lambda^1_0, \dots, \lambda^1_{b-1})$-design. Furthermore, if instead we fix the first row of those tabloids, we have an integral $(b+c, c, \lambda^2_0, \dots, \lambda^2_{c-1})$-design. Therefore, by Corollary~\ref{dim},
\begin{align*}
\text{dim}_{\mathbb{Z}}\bigcap_{v = 0}^{b-1}\psi_{1,v}^{-1}(\mathbb{Z}f_{(a+b-v, v, c)}) &= \binom{a+b+c}{c} \left( \binom{a+b}{b} - \binom{a+b}{b-1}+1\right),\\
\text{dim}_{\mathbb{Z}}\bigcap_{w=0}^{c-1}\psi_{2,w}^{-1}(\mathbb{Z}f_{(a, b+c -w, w)}) &= \binom{a+b+c}{a} \left( \binom{b+c}{c} - \binom{b+c}{c-1}+1\right).
\end{align*}

\noindent
Clearly, we have the required $u$ if the system of integral designs coming from $\psi_{1,v}$'s overlaps with those coming from $\psi_{1,w}$'s non-trivially i.e. 
\[(
\cap_{v = 0}^{b-1}\psi_{1,v}^{-1}(\mathbb{Z}f_{(a+b-v, v, c)})) \bigcap (\cap_{w=0}^{c-1}\psi_{2,w}^{-1}(\mathbb{Z}f_{(a, b+c -w, w)})) \text{ strictly contains $S^{(a, b, c)}$}.
\] This will be true if
\[
\text{dim}_{\mathbb{Z}} S^{(a, b, c)} < \text{dim}_{\mathbb{Z}}\bigcap_{v = 0}^{b-1}\psi_{1,v}^{-1}(k) + \text{dim}_{\mathbb{Z}}\bigcap_{w=0}^{c-1}\psi_{2,w}^{-1}(k) - \text{dim}_{\mathbb{Z}} M^{(a, b ,c)}.
\]

\begin{Theorem}\label{3part}
We have
\[
\dim_{\mathbb{Z}} S^{(a, b, 1)} < \dim_{\mathbb{Z}}\bigcap_{v = 0}^{b-1}\psi_{1,v}^{-1}(k) + \dim_{\mathbb{Z}}\psi_{2,0}^{-1}(k) - \dim_{\mathbb{Z}} M^{(a, b ,1)},
\] 
i.e. \textbf{Problem~\ref{non-vanishing hom}} can be solved for partitions of the form $(a,b,1)$ .\end{Theorem}
\begin{proof}
Let $d = a+b+c$. By a theorem of Frame, Robinson and Thrall~\cite{FRT},
\begin{align*}
&\text{dim}_{\mathbb{Z}}S^{(a, b, c)} = \frac{d!}{\prod (\text{hook lengths in} [(a, b, c)])} \\
&= \frac{d!}{(a+2)\dots (a-c+3)(a-c+1)\dots (a-b+2)(a-b)! (b+1)\dots (b-c+2)(b-c)! c!}\\
&=\frac{d!(a-c+2)(a-b+1)(b-c+1)}{(a+2)!(b+1)!c!}.
\end{align*}
Also, by $4.2$ in James~\cite{James78},
\[
\text{dim}_{\mathbb{Z}} M^{(a, b, c)} = \frac{d!}{a!b!c!}
\]
 We have
\begin{align*}
&\frac{d!(a-c+2)(a-b+1)(b-c+1)}{(a+2)!(b+1)!c!} + \frac{d!}{a!b!c!} \\
&<  \binom{d}{c} \left( \binom{a+b}{b} - \binom{a+b}{b-1}+1\right) + \binom{d}{a} \left( \binom{b+c}{c} - \binom{b+c}{c-1}+1\right)\\
\iff &\frac{d!(a-c+2)(a-b+1)(b-c+1)}{(a+2)!(b+1)!c!} + \frac{d!}{a!b!c!}+\binom{d}{c}\binom{a+b}{b-1}+\binom{d}{a}\binom{b+c}{c-1}\\
&<\binom{d}{c}\binom{a+b}{b}+\binom{d}{a}\binom{b+c}{c}+\binom{d}{c}+\binom{d}{a}\\
\iff &\frac{d!(a-c+2)(a-b+1)(b-c+1)}{(a+2)!(b+1)!c!} + \frac{d!}{a!b!c!} +\frac{d!}{c!(a+b)!}\frac{(a+b)!}{(b-1)!(a+1)!} + \frac{d!}{a!(b+c)!} \cdot\\
&\cdot\frac{(b+c!)}{(c-1)!(b+1)!}
<\frac{d!}{c!(a+b)!}\frac{(a+b)!}{b!a!}+\frac{d!}{a!(b+c)!}\frac{(b+c)!}{b!c!}+\frac{d!}{c!(a+b)!}+\frac{d!}{a!(b+c)!}\\
\iff & \frac{1}{c!(a+1)!(b-1)!}+\frac{1}{a!(c-1)!(b+1)!}+\frac{(a-c+2)(a-b+1)(b-c+1)}{(a+2)!(b+1)!c!}\\
<&\frac{1}{a!b!c!}+\frac{1}{c!(a+b)!}+\frac{1}{a!(b+c)!}\\
\iff &\frac{1}{(a+1)!(b-1)!}+\frac{c}{a!(b+1)!}+\frac{(a-c+2)(a-b+!)(b-c+1)}{(a+2)!(b+1)!}\\
<&\frac{1}{a!b!} +\frac{1}{(a+b)!}+\frac{c!}{a!(b+c)!}\\
\iff &\frac{b}{a+1}+\frac{c}{b+1}+\frac{(a-c+2)(a-b+1)(b-c+1)}{(a+2)(a+1)(b+1)}
< 1 +\frac{a!b!}{(a+b)!}+\frac{c!b!}{(b+c)!} \label{dagger}\tag{$\dagger$}
\end{align*}
\noindent
If $c =1$, then \eqref{dagger} becomes
\[
\frac{b}{a+1}+\frac{1}{b+1}+\frac{(a+1)(a-b+1)b}{(a+2)(a+1)(b+1)} < 1 + \frac{a!b!}{(a+b)!}+\frac{b!}{(b+1)!}.
\]
This is equivalent to
\[
\frac{(a+1)(a-b+1)b}{(a+2)(a+1)(b+1)} < \frac{a-b+1}{a+1} + \frac{a!b!}{(a+b)!},
\]
\noindent
which is clearly true since $\frac{(a+1)b}{(a+2)(b+1)} < 1$, so we are done.
\end{proof}

\section{Conclusion}
In this paper, we aimed to construct the desired element $u$ of the permutation module $M^\lambda$ of the symmetric group $\mathfrak{S}_n$, satisfying the necessary and sufficient conditions found by Hemmer~\cite{Hemmer12} for the cohomology group $H^1(\mathfrak{S}_n, S^\lambda)$ to be non-zero by specializing first to two-part partitions. This leads to an unexpected and surprising connection of combinatorial $t$-design theory to the representation theory of the symmetric groups. 

The next natural step would be to generalize the method by Graver-Jurkat to solve our problem for arbitrary partitions where now instead of one integral design, we have a system of integral designs. For example, for $3$-part partitions $(a, b, c)$, we have seen that fixing the third row of the tabloids involved in $u$ gives an $(a+b,b,\lambda_0,...,\lambda_{b-1})$-design for $\psi_{1, v}$, while fixing the first row of tabloids involved in $u$ gives a $(b + c, c, \alpha_0, . . . , \alpha_{c-1}$)-design for $\psi_{2,v}$. Note that constructing a $u$ for $s$-part partitions from a $u$ for $(s-1)$-part partitions satisfied conditions of \textbf{\emph{Problem 1}} is equivalent to constructing a map
\[
\Theta: H^1(\mathfrak{S}_n, S^\lambda) \to H^1(\mathfrak{S}_{n+a}, S^{(a, \lambda_1,...,\lambda_{s-1})}), 
\]
\noindent
where  $\lambda = (\lambda_1, . . . , \lambda_{s-1}) \vdash n$ and $a \equiv (-1) \pmod{p^{l_p(\lambda_1)}}$.

Thus it appears that this approach can be used to tackle the case $i = 1$ in the following problem posed by Hemmer~\cite{Hemmer12}:

\begin{Problem}
Does the isomorphism $H^i(\mathfrak{S}_n, S^\lambda) \cong H^i(\mathfrak{S}_{n+a}, S^{(a,\lambda_1,...,\lambda_{s-1})})$ hold for $i > 0$?
\end{Problem}

Alternatively, one could investigate the space $\bigcap_{v=0}^{\lambda_i-1} \psi^{-1}(k)$ and try to recast Graver and Jurkat's construction in terms of tabloids and mimic James's formulation of the Specht modules as the kernel intersection. If this is successful, we will then be able find all
u satisfying conditions (i) in \textbf{Theorem~\ref{ext}}. Combining this with Weber's work ~\cite{Weber}, we will be able to find when $H^1(\mathfrak{S}_n, S^\lambda)\neq 0$ for all partitions with at least five parts.

Another problem could perhaps be tackled using this approach is

\begin{Problem} [Hemmer ~\cite{Hemmer12}] Let $\lambda \vdash n$ and let $p > 2$. Then there is an isomorphism 
\[
H^1(\mathfrak{S}_{pn}, S^{p\lambda}) \cong H^1(\mathfrak{S}_{p^2n},S^{p^2\lambda})\ (*),
\]
 where for a partition $\lambda = (\lambda_1, \lambda_2, \dots )$, define $p\lambda = (p\lambda_1, p\lambda_2, \dots) \vdash pn$ . Suppose $H^1(\mathfrak{S}_{pn}, S^{p\lambda})\neq 0$ and suppose one has constructed $u \in M^{p\lambda}$ satisfying Theorem~\ref{ext}, describe a general method to construct  $\tilde{u} \in M^{p^2\lambda}$ corresponding to an element in $H^1(\mathfrak{S}_{p^2n}, S^{p^2\lambda})$ and realizing the isomorphism $(*)$.
\end{Problem}

\section{Acknowledgements}
This paper is based on results of my PhD thesis~\cite{HaThuThesis} prepared at the University of Cambridge. I would like to thank my supervisor Dr Stuart Martin for his encouragement and support.

\bibliography{HTreferences}
\bibliographystyle{plain}
\end{document}